\documentclass[12pt]{amsart}

\usepackage{amsmath,amssymb,amsthm,mathtools,mathdots,nicefrac,tikz,bbding,stmaryrd,float,multicol,caption,bbm}
\usepackage[mathscr]{euscript}
\usepackage[margin=1in, marginparwidth=1.75cm]{geometry}
\usepackage{color}
\usepackage{hyperref, float}
\hypersetup{colorlinks=true,citecolor=purple,linkcolor=purple}

\theoremstyle{definition}
\newtheorem{theorem}{Theorem}[section]

\newtheorem{lemma}[theorem]{Lemma}
\newtheorem{proposition}[theorem]{Proposition}
\newtheorem{definition}[theorem]{Definition}
\newtheorem{example}[theorem]{Example}
\newtheorem{corollary}[theorem]{Corollary}

\newcommand{\Z}{\mathbb{Z}}
\newcommand{\R}{\mathbb{R}}

\newcommand{\Q}{\mathbb{Q}}
\newcommand{\frakp}{\mathfrak{p}}

\newcommand{\Spec}{\text{Spec}}
\newcommand{\Ass}{\text{Ass}}

\newcommand{\Sing}{\text{Sing}}
\newcommand{\Min}{\text{Min}}
\newcommand{\depth}{\text{depth }}
\newcommand{\height}{\text{ht}}

\newtheorem{manualtheoreminner}{Theorem}
\newenvironment{manualtheorem}[1]{%
  \renewcommand\themanualtheoreminner{#1}%
  \manualtheoreminner
}{\endmanualtheoreminner}

\newcommand{\interior}[1]{%
  {\kern0pt#1}^{\mathrm{o}}%
}

\DeclarePairedDelimiter\abs{\lvert}{\rvert}

\usepackage[backrefs]{amsrefs}

\title{Completions of Countable Excellent Domains and Countable Noncatenary Domains}
\author{S. Loepp and Teresa Yu}
\begin{document}

\maketitle

\begin{abstract}
    We find necessary and sufficient conditions for a complete local ring containing the rationals to be the completion of a countable excellent local (Noetherian) domain. Furthermore, we find necessary and sufficient conditions for a complete local ring to be the completion of a countable noncatenary local domain, as well as necessary and sufficient conditions for it to be the completion of a countable noncatenary local unique factorization domain.
\end{abstract}

\section{Introduction}

While the structure of complete local (Noetherian) rings is well-understood via Cohen's Structure Theorem, the structure of local rings that are not complete is much more mysterious. This asymmetry in understanding has spurred results focusing on the relationship between a local ring and its completion. Past results have characterized complete local rings that are completions of local rings with specific ring properties, such as being an integral domain (see \cite{lech}), being an excellent domain in the characteristic zero case (see \cite{loepp03}), and being a noncatenary domain or a noncatenary unique factorization domain (see \cite{SMALL17}). The authors of \cite{SMALL19} focus on characterizing completions of local domains with certain cardinalities. In particular, they characterize complete local rings that are the completion of a countable local domain. In this paper, we first extend this result to countable {\em excellent} local domains in the case where the complete local ring contains the rationals. In particular, in Section \ref{sec:countable_excellent}, we prove the following result.


\begin{manualtheorem}{\ref{thm:countable_excellent_dimgeq1_char}}
Let $T$ be a complete local ring with maximal ideal $M$ and suppose that $T$ contains the rationals. Then $T$ is the completion of a countable excellent local domain if and only if the following conditions hold:
\begin{enumerate}
    \item $T$ is equidimensional,
    \item $T$ is reduced, and
    \item $T/M$ is countable.
\end{enumerate}
\end{manualtheorem}

Using similar techniques that we use to prove this result, we characterize complete local rings that are the completion of a countable noncatenary local domain in Section~\ref{sec:noncat}. In particular, we prove the following theorem.

\begin{manualtheorem}{\ref{thm:countable_noncat_domain_char}}
Let $T$ be a complete local ring with maximal ideal $M$. Then $T$ is the completion of a countable noncatenary local domain if and only if the following conditions hold:
\begin{enumerate}
    \item no integer of $T$ is a zero divisor,
    \item $M\notin\Ass(T)$,
    \item there exists $P\in\Min(T)$ such that $1<\dim(T/P)<\dim T$, and
    \item $T/M$ is countable.
\end{enumerate}
\end{manualtheorem}

Finally, in Section \ref{sec:noncat}, we prove the following result characterizing completions of countable noncatenary local unique factorization domains.

\begin{manualtheorem}{\ref{thm:noncat_ufd_char}}
Let $T$ be a complete local ring with maximal ideal $M$. Then $T$ is the completion of a countable noncatenary local unique factorization domain if and only if the the following conditions hold:
\begin{enumerate}
    \item no integer of $T$ is a zero divisor,
    \item $\depth T>1$,
    \item there exists $P\in\Min(T)$ such that $2<\dim(T/P)<\dim T$, and
    \item $T/M$ is countable.
\end{enumerate}
\end{manualtheorem}

Interestingly, one can use these results to characterize complete local rings containing the rationals that are the completion of an uncountable excellent local domain with a countable spectrum, complete local rings that are the completion of an uncountable noncatenary local domain with a countable spectrum, and complete local rings that are the completion of an uncountable noncatenary local unique factorization domain with a countable spectrum.  These results appear in the forthcoming paper, ``Completions of Uncountable Local Rings with Countable Spectra," which will be posted on the arXiv soon.

The outline of this paper is as follows. In Section~\ref{sec:background}, we provide background on complete local rings and excellent rings. In Section~\ref{sec:countable_excellent}, we identify necessary and sufficient conditions on a complete local ring containing the rationals to be the completion of a countable excellent local domain. Using the techniques developed in this section, we then characterize completions of countable noncatenary local domains as well as completions of countable noncatenary local unique factorization domains in Section~\ref{sec:noncat}.

\section{Background}\label{sec:background}

Throughout this paper, all rings are commutative with unity. We say a ring is \textit{quasi-local} if it has a unique maximal ideal, and we say that it is \textit{local} if it is both quasi-local and Noetherian. We denote a quasi-local ring $R$ with unique maximal ideal $M$ by $(R,M)$, and we use $\widehat{R}$ to denote the completion of $R$ with respect to its maximal ideal when $R$ is local.  Finally, we use the standard abbreviation UFD to denote a unique factorization domain and the standard abbreviations $\mathbb{Z}$ to denote the integers, $\mathbb{Q}$ to denote the rationals, and $\mathbb{R}$ to denote the reals.

In \cite{lech}, Lech characterizes completions of integral domains by proving the following result.

\begin{theorem}[\cite{lech}, Theorem 1]\label{thm:lech_char_domain}
A complete local ring $(T,M)$ is the completion of a local domain if and only if
\begin{enumerate}
    \item no integer of $T$ is a zero divisor, and
    \item unless equal to $(0)$, $M\notin\Ass(T)$.
\end{enumerate}
\end{theorem}

More recently in \cite{SMALL19}, the authors characterize completions of {\em countable} local domains. We use this result to construct a base ring, from which we begin our constructions of countable excellent local domains and countable noncatenary local domains.

\begin{theorem}[\cite{SMALL19}, Corollary 2.15]\label{prop:small19_countable_domain}
Let $(T,M)$ be a complete local ring. Then $T$ is the completion of a countable local domain if and only if
\begin{enumerate}
    \item no integer is a zero divisor of $T$,
    \item unless equal to $(0)$, $M\notin\Ass (T)$, and
    \item $T/M$ is countable.
\end{enumerate}
\end{theorem}

Throughout this paper, we often need to show that a ring has a certain completion. In order to do this, we make use of the following proposition, which gives sufficient conditions.

\begin{proposition}[\cite{heitmann}, Proposition 1]\label{prop:complete_machine}
If $(R,R\cap M)$ is a quasi-local subring of a complete local ring $(T,M)$, the map $R\to T/M^2$ is onto, and $IT\cap R=IR$ for every finitely generated ideal $I$ of $R$, then $R$ is Noetherian and the natural homomorphism $\widehat{R}\to T$ is an isomorphism.
\end{proposition}

Note that if $R$ is a local ring and $\widehat{R}=T$, then $T$ is a faithfully flat extension of $R$. It follows that $R$ and $T$ satisfy the going-down theorem, implying that if $P\in\Spec(T)$, then $\height(P\cap R)\leq\height(P)$. Furthermore, the fact that $T$ is a faithfully flat extension of $R$ implies that for any finitely generated ideal $I$ of $R$, we have that $IT\cap R=IR$. This allows us to show that the converse of Proposition~\ref{prop:complete_machine} also holds.

\begin{proposition}\label{prop:completion_machine_converse}
If $(T,M)$ is a complete local ring and $(R,R\cap M)$ is a local subring of $T$ such that $\widehat{R}=T$, then the map $R\to T/M^2$ is onto, and $IT\cap R=IR$ for every finitely generated ideal $I$ of $R$.
\end{proposition}

\begin{proof}
We first show that since $\widehat{R}=T$, the map $R\to T/M^2$ is onto. 
Let $t + M^2 \in T/M^2$.  Then, since $\widehat{R}=T$, we have that $t = r_0 + r_1 + r_2 + \cdots$ where $r_i \in (R \cap M)^i$ for all $i \geq 1$.  Therefore, $t + M^2 = r_0 + r_1 + M^2$, and we have that  $r_0 + r_1 \in R$ maps to $t + M^2$ under the natural map from $R$ to $T/M^2$.

Since $\widehat{R}=T$, the ring $T$ is a faithfully flat extension of $R$, implying that $IT\cap R=IR$ for any finitely generated ideal $I$ of $R$.
\end{proof}

We then have the following corollary, which provides another way to show that a ring has a certain completion.

\begin{corollary}\label{cor:basering_complete_macine}
Suppose $(T,M)$ is a complete local ring, and $(R,R\cap M)$ is a local subring of $T$ such that $\widehat{R}=T$. Let $(A,A\cap M)$ be a quasi-local subring of $T$ such that $R\subseteq A$, and such that, for every finitely generated ideal $I$ of $A$, $IT\cap A=IA$. Then, $A$ is Noetherian and $\widehat{A}=T$.
\end{corollary}

\begin{proof}
By Proposition~\ref{prop:completion_machine_converse}, we have that the map $R\to T/M^2$ is onto. It follows that the map $A\to T/M^2$, is onto. By assumption, for every finitely generated ideal $I$ of $A$, $IT\cap A=IA$. Thus, by Proposition~\ref{prop:complete_machine}, $A$ is Noetherian with completion $T$.
\end{proof}

We now provide some background on excellent rings, starting with a few definitions. Recall that a ring $A$ is \textit{catenary} if, for any pair of prime ideals $P\subsetneq Q$ of $A$, all saturated chains of prime ideals between $P$ and $Q$ have the same length. If a ring is not catenary, then it is called \textit{noncatenary}. A Noetherian ring $A$ is \textit{universally catenary} if $A[x_1,\ldots,x_n]$ is catenary for every $n\geq 0$; although this is not the classical definition, it is equivalent to the classical definition (see \cite{matsumura}, pg. 118). For any $P\in\Spec(A)$, define $k(P)\coloneqq A_P/PA_P$.

\begin{definition}[\cite{rotthaus}, Definition 1.4]\label{defn:excellent}
A local ring $A$ is {\em excellent} if
\begin{itemize}
    \item[(a)] for all $P\in\Spec(A)$, $\widehat{A}\otimes_A L$ is regular for every finite field extension $L$ of $k(P)$, and
    \item[(b)] $A$ is universally catenary.
\end{itemize}
\end{definition}

A local ring is \textit{formally equidimensional} if its completion is equidimensional. A consequence of Theorem 31.6 from \cite{matsumura} is that a formally equidimensional Noetherian local ring is universally catenary. Thus, we have the following result from \cite{loepp2018uncountable}.

\begin{theorem}[\cite{loepp2018uncountable}, Theorem 2.4]\label{thm:equi_implies_univcat}
Let $A$ be a local ring such that its completion, $\widehat{A}$, is equidimensional. Then $A$ is a universally catenary.
\end{theorem}

It is noted in \cite{rotthaus} that, for Definition \ref{defn:excellent}, it is enough to only consider the purely inseparable finite field extensions  $L$ of $k(P)$. Because of this, we have the following modification of \cite[Lemma 2.5]{loepp2018uncountable} that gives sufficient criteria for a subring of a complete local ring satisfying certain conditions to be excellent. The proof is almost verbatim from the proof given in \cite{loepp2018uncountable}.

\begin{lemma}\label{lem:excellent_sufficient_criteria}
Let $(T,M)$ be a complete local ring that is equidimensional and suppose $\Q\subseteq T$. Given a subring $(A,A\cap M)$ of $T$ with $\widehat{A}=T$, $A$ is excellent if, for every $P\in\Spec(A)$ and for every $Q\in\Spec(T)$ with $Q\cap A=P$, $(T/PT)_Q$ is a regular local ring.
\end{lemma}

\begin{proof}
We know that $A$ is a local ring, so we must show now that both conditions of Definition~\ref{defn:excellent} hold. By Theorem~\ref{thm:equi_implies_univcat}, we have that $A$ is universally catenary. Thus, it remains to consider $T\otimes_A L$ for every purely inseparable finite field extension $L$ of $k(P)$ for each $P\in\Spec(A)$. Since $\Z\subseteq A$ and all nonzero integers are units, we have that $\Q\subseteq k(P)$, so $k(P)$ has characteristic $0$. Every finite field extension with characteristic $0$ is separable. Since it is sufficient to check only purely inseparable field extensions, this leaves only the trivial field extension, as this is the only field extension that is both separable and purely inseparable. Thus, we need only show that $T\otimes_A k(P)$ is regular for every $P\in\Spec(A)$. Note that, for $Q\in\Spec(T)$ with $Q\cap A=P$, the ring $T\otimes_A k(P)$ localized at $Q\otimes k(P)$ is isomorphic to $(T/PT)_Q$. Thus, it suffices to show that $(T/PT)_Q$ is a regular local ring. 
\end{proof}

The following result from \cite{rotthaus} concerns the structure of $\Sing(R)$ for excellent rings.

\begin{lemma}[\cite{rotthaus}, Corollary 1.6]\label{lem:sing_closed}
If $R$ is excellent, then $\Sing(R)$ is closed in the Zariski topology, i.e., $\Sing(R)=V(I)$ for some ideal $I$ of $R$.
\end{lemma}

We end this section with two results on the cardinalities of local rings and their quotient rings. 

\begin{proposition}\label{prop:t/m^2_countable}
Let $(T,M)$ be a local ring. If $T/M$ is finite, then $T/M^2$ is finite. If $T/M$ is infinite, then $\abs{T/M^2}=\abs{T/M}$.
\end{proposition}

\begin{proof} The result follows from Lemma 2.12 in \cite{SMALL19}.
\end{proof}

Let $c$ denote the cardinality of $\R$.

\begin{lemma}[\cite{dundon}, Lemma 2.2]\label{lem:t_mod_p_uncountable}
Let $(T,M)$ be a complete local ring with $\dim T\geq 1$. Let $P$ be a nonmaximal prime ideal of $T$. Then, $\abs{T/P}=\abs{T}\geq c$.
\end{lemma}

\section{Completions of Countable Excellent Local
Domains}\label{sec:countable_excellent}

In this section, we give necessary and sufficient conditions for a complete local ring containing $\Q$ to be the completion of a countable excellent local domain. 

We first cite the following result from \cite{loepp03}, which gives necessary and sufficient conditions on a complete local ring containing the integers to be the completion of an excellent local domain.

\begin{theorem}[\cite{loepp03}\label{excellentdomain}, Theorem 9]\label{thm:loepp_excellent_domain_char}
Let $(T,M)$ be a complete local ring containing the integers. Then $T$ is the completion of a local excellent domain if and only if it is reduced, equidimensional, and no integer of $T$ is a zero divisor.
\end{theorem}

To prove that the conditions in Theorem \ref{excellentdomain} are sufficient, the author of \cite{loepp03} assumes that the conditions hold on a complete local ring $T$ and then constructs a subring $A$ of $T$ such that $A$ is an excellent domain with $\widehat{A} = T$.  If the dimension of $T$ is at least two, the ring $A$ for that construction has uncountably many prime ideals, and so $A$ is uncountable.  As we want to construct a {\em countable} excellent domain, the construction in \cite{loepp03} does not work.  The proof of our main result in this section, then, is fundamentally different than the proof of Theorem 9 in \cite{loepp03}.

The conditions in Theorem \ref{excellentdomain}, along with conditions (1), (2), and (3) of Theorem~\ref{prop:small19_countable_domain}, are, of course, necessary for a complete local ring containing the rationals to be the completion of a countable excellent local domain. Note that in the case that $\dim T\geq 1$, $T$ being reduced is a stronger condition than $M\notin\Ass(T)$, since if $T$ is a Noetherian reduced ring, then $\Min(T)=\Ass(T)$. In addition, if a complete local ring contains the rationals, then every nonzero integer is a unit, and thus not a zero divisor. It is therefore enough to consider the conditions of being equidimensional, reduced, and having a countable residue field. The bulk of this section is dedicated to showing that these conditions are also sufficient in the case that $\dim T\geq 1$.

We accomplish this by showing that if a complete local ring satisfies the aforementioned conditions, then one can construct a countable excellent local domain whose completion is the initial complete local ring. In order to construct the ring, we adjoin carefully-chosen elements to a subring while ensuring that important properties are preserved. We introduce a definition that is modified from \cite[Definition 2.2]{SMALL19} which encapsulates these desired properties that need to be preserved.

\begin{definition}\label{defn:br_0_subring}
Let $(T,M)$ be a complete local ring and let $(R_0,R_0\cap M)$ be a countable local subring of $T$ such that $R_0$ is a domain and $\widehat{R}_0=T$. Let $(R,R\cap M)$ be a quasi-local subring of $T$ with $R_0\subseteq R$. Suppose that
\begin{itemize}
    \item[(a)] $R$ is countable, and
    \item[(b)] $R\cap P=(0)$ for every $P\in\Ass(T)$.
\end{itemize}
Then we call $R$ a
\textit{built-from-}$R_0$ subring of $T$, or a $BR_0$-subring of $T$ for short.
\end{definition}

In other words, $BR_0$-subrings of $T$ are countable quasi-local rings $(R,R\cap M)$ such that $R_0\subseteq R\subseteq T$ and $R$ contains no zero divisors of $T$. Note that the countable union of an ascending chain of $BR_0$-subrings of $T$ is also a $BR_0$-subring of $T$.

The general outline of our construction of a countable excellent local domain is as follows. Let $(T,M)$ be a complete local ring such that $\Q\subseteq T$ and $\dim T\geq 1$, and further suppose that $T/M$ is countable, $T$ is reduced, and $T$ is equidimensional. Beginning with a countable local domain $R_0$ whose completion is $T$, we construct an ascending chain of $BR_0$-subrings of $T$, all with completion $T$. For each $BR_0$-subring in our chain, to get the next $BR_0$-subring in the chain, we first adjoin generators of certain prime ideals of $T$, whose properties are described in Lemma~\ref{lem:Q_J_countable}. We adjoin the generators in a way detailed in Lemma~\ref{lem:adjoin_gen_set} so that the resulting ring is indeed a $BR_0$-subring of $T$. We then use Lemma~\ref{lem:construct_precompletion_from_pbsubring} to show that one can construct a countable local domain from this $BR_0$-subring whose completion is $T$. This is the next ring in our ascending chain. The union of this ascending chain of rings is excellent, as shown in Theorem~\ref{thm:countable_excellent_domain_existence}. The union is also countable and has compleition $T$, so this is our desired countable excellent local domain.

When constructing our countable excellent domain, we adjoin elements of $T$ to $BR_0$-subrings so that the resulting ring is also a $BR_0$-subring of $T$. The next lemma, adapted from \cite[Lemma 11]{loepp97}, provides sufficient conditions on the elements that we are able to adjoin.

\begin{lemma}\label{lem:pb_subring_adjoin}
Suppose $(T,M)$ is a complete local ring and $(R_0,R_0\cap M)$ is a countable local domain with $R_0\subseteq T$ and $\widehat{R}_0=T$. Let $C$ be the maximal elements of $\Ass(T)$.  Let $(R,R\cap M)$ be a $BR_0$-subring of $T$, and let $x\in T$ such that, for all $P\in C$, we have that 
$x+P$ is transcendental over $R/(R\cap P) \cong R$ as an element of $T/P$. Then, $R'=R[x]_{R[x]\cap M}$ is a $BR_0$-subring of $T$.
\end{lemma}

\begin{proof}
First notice that since $R$ is countable and $R_0\subseteq R$, we have that $R'$ is countable and $R_0\subseteq R'$. 
We now show that if $P'\in\Ass(T)$, then $R'\cap P'=(0)$. It suffices to show that $R[x]\cap P'=(0)$. Let $P \in C$ such that $P' \subseteq P$.  Suppose $u\in R[x]\cap P'$.  Then $u \in R[x]\cap P$, and $u$ is of the form $u=a_nx^n+\cdots +a_1x+a_0$ with $a_i\in R$. Notice that $R/(R\cap P) \cong R$ injects into $T/P$, so we can view $R$ as a subring of $T/P$. Thus, since $u\in P$ and
$x+P$ is transcendental over $R$ as an element of $T/P$, $a_i\in P$ for all $i$. But $a_i\in R\cap P=(0)$. Thus, $u=0$, so $R'\cap P'=(0)$, and $R'$ is a $BR_0$-subring of $T$.
\end{proof}

In the next lemma, which is adapted from \cite{loepp2018uncountable}, we describe a certain set of prime ideals. Our goal is to adjoin generators of such prime ideals to the ring we are constructing, as this will allow us to show that our final ring is excellent.

\begin{lemma}\label{lem:Q_J_countable}
Let $(T,M)$ be a complete local reduced ring and let $(R,R\cap M)$ be a countable local domain with $R\subseteq T$ and $\widehat{R}=T$. Then,
\[\bigcup_{P\in\Spec(R)}\{Q\in\Spec(T)\mid Q\in\min I\text{ for $I$ where }\Sing(T/PT)=V(I/PT)\}\]
is a countable set. Furthermore, for any prime ideal $Q$ in this set, $Q\nsubseteq \frakp$ for all $\frakp\in \Ass(T)$.
\end{lemma}

\begin{proof}
Since $R$ is countable and Noetherian, $\Spec(R)$ is countable. Thus, it suffices to show that the set is countable with respect to any fixed $P\in\Spec(R)$. Let $P\in\Spec(R)$. Since $T$ is a complete local ring, $T/PT$ is excellent. By Lemma~\ref{lem:sing_closed}, $\Sing(T/PT)=V(I/PT)$ for some ideal $I$ of $T$ containing $PT$. Consider the set of minimal prime ideals $Q$ of $I$. Since $T$ is Noetherian, this set is finite.

We now show that for any $Q$ in this set, we have that $Q\nsubseteq \mathfrak{p}$ for all $\mathfrak{p}\in\Ass(T)$. Fix $P\in\Spec(R)$, and consider $\mathcal{C}=\{Q\in\Spec(T)\mid Q\in\min I\text{ for $I$ where }\Sing(T/PT)=V(I/PT)\}$. First suppose $\text{ht}(P)\geq 1$ in $R$. Then, since $R$ is a domain and $\widehat{R}=T$, there exists $a\in P$ such that $a$ is not a zero divisor of $T$, i.e., $a\notin \mathfrak{p}$ for all $\mathfrak{p}\in\Ass(T)$. Since $PT\subseteq I\subseteq Q$, we have that $a\in Q$; thus, $Q\nsubseteq \mathfrak{p}$ for all associated prime ideals $\mathfrak{p}$ of $T$.

Now suppose $\text{ht}(P)=0$. Since $R$ is a domain, this means that $P=(0)$. Then, we are considering $I$ such that $\Sing(T)=V(I)$. 
Since $T$ is reduced, it satisfies Serre's $(R_0)$ condition, meaning for all $\mathfrak{q}\in\Spec(T)$ with $\height(\mathfrak{q})=0$, we have that $T_{\mathfrak{q}}$ is regular. Thus, it cannot be that $I$ is contained in a prime ideal of height $0$; otherwise, such a prime ideal would be in the singular locus of $T$, which is impossible. Notice that $\Min(T)=\Ass(T)$ since $T$ is a reduced Noetherian ring; this implies that $\text{ht}(\frakp)=0$ for all $\frakp\in\Ass(T)$. Thus, it must be that $I\nsubseteq \frakp$ for all $\frakp\in\Ass(T)$. Since $I\subseteq Q$, we also have that $Q\nsubseteq \frakp$ for all $\frakp\in\Ass(T)$.
\end{proof}

Next, we adjoin generating sets for the prime ideals described in the previous lemma. In order to accomplish this, we make use of the following result, which is a stronger version of the Prime Avoidance Theorem.  In particular, it enables us to avoid cosets of certain prime ideals.

\begin{lemma}[\cite{pippa}, Lemma 2.4]\label{lem:pippa_prime_avoidance}
Let $(T,M)$ be a complete local ring such that $\dim T\geq 1$, let $C$ be a finite set of nonmaximal prime ideals of $T$ such that no ideal in $C$ is contained in another ideal of $C$, and let $D$ be a subset of $T$ such that $\abs{D}<\abs{T}$. Let $I$ be an ideal of $T$ (not necessarily a proper ideal) such that $I\nsubseteq P$ for all $P\in C$. Then $I\nsubseteq\bigcup\{r+P\mid P\in C, r\in D\}$.
\end{lemma}

The following lemma is inspired by a procedure from \cite[Lemma 3.6]{SMALL17}. It describes a method for adjoining a generating set for a prime ideal to a $BR_0$-subring to obtain a larger $BR_0$-subring. 

\begin{lemma}\label{lem:adjoin_gen_set}
Suppose $(T,M)$ is a complete local ring with $\dim T\geq 1$, $(R_0,R_0\cap M)$ is a countable local domain with $R_0\subseteq T$ and $\widehat{R}_0=T$, and $(R, R \cap M)$ is a $BR_0$-subring of $T$. Let $Q\in\Spec(T)$ such that $Q\nsubseteq P$ for all $P\in\Ass(T)$. Then there exists a $BR_0$-subring of $T$, $(R', R' \cap M)$, such that $R\subseteq R'$ and $R'$ contains a generating set for $Q$.
\end{lemma}

\begin{proof}
Let $(x_1,\ldots,x_n)$ be a generating set for $Q$. We inductively define a chain of $BR_0$-subrings of $T$, $R=R_1\subseteq R_2\subseteq\cdots\subseteq R_{n+1}$ such that $R_{n+1}$ contains a generating set for $Q$. To construct $R_{i+1}$ from $R_i$, we show that there exists an element  $\tilde{x}_i$ of $T$ so that $R_{i+1}\coloneqq R_{i}[\tilde{x}_i]_{(R_{i}[\tilde{x}_i]\cap M)}$ is a $BR_0$-subring of $T$ and such that we can replace $x_i$ in the generating set of $Q$ with $\tilde{x}_i$. By assumption and the Prime Avoidance Theorem, $Q\nsubseteq \bigcup_{P\in\Ass(T)}P$; thus, there exists $y\in Q$ such that $y\notin P$ for all $P\in\Ass(T)$. 


We first construct $R_2$ by considering $R_1=R$. We find $\tilde{x}_1=x_1+\alpha_1y$ with $\alpha_1\in M$ so that $\tilde{x}_1+P$ is transcendental over $R_1/(R_1\cap P) = R/(R\cap P) \cong R$ as an element of $T/P$ for all $P$ maximal in $\Ass(T)$. 
To do this, first fix $P$, a maximal element of $\Ass(T)$, and consider $x_1+ty+P$ for some $t\in T$. Notice that since $R_1/(R_1\cap P)$ is countable, its algebraic closure in $T/P$ is also countable. 
In addition, each choice of $t+P$ gives a different $x_1+ty+P$, since $y\notin P$. Thus, for at most countably many choices of $t+P$, the element $x_1+ty + P$ of $T/P$ is algebraic over $R_1/(R_1\cap P)$.

Let $D_{(P)}$ be a full set of coset representatives of elements $t+P$ of $T/P$ that make $x_1+ty+P$ algebraic over $R_1/(R_1\cap P)$. Let $C$ be the maximal elements of $\Ass(T)$ and let $D_1=\bigcup_{P\in C}D_{(P)}$. Then, $D_1$ is countable since there are finitely many associated prime ideals $P$ of $T$ and $D_{(P)}$ is countable for each $P$. Since $\dim T\geq 1$, $T$ is uncountable by Lemma~\ref{lem:t_mod_p_uncountable}. It follows that $\abs{D_1}<\abs{T}$.  In addition, $T$ is the completion of a local domain and $\dim T\geq 1$, so $M\notin\Ass(T)$ by Theorem~\ref{thm:lech_char_domain}; then, $M\nsubseteq P$ for all $P\in C$. Applying Lemma~\ref{lem:pippa_prime_avoidance} with $I=M$, $C$ the maximal elements of $\Ass(T)$, and $D=D_1$, we have that
\[M\nsubseteq\bigcup\{r+P\mid r\in D_1,P\in C\},\]
so there exists $\alpha_1\in M$ such that $x_1+\alpha_1 y+P$ is transcendental over $R_1/(R_1\cap P)$ for every $P\in C$. Let $\tilde{x}_1\coloneqq x_1+\alpha_1y$. By Lemma~\ref{lem:pb_subring_adjoin}, we have that $R_2\coloneqq R_1[\tilde{x}_1]_{(R_1[\tilde{x}_1]\cap M)}$ is a $BR_0$-subring of $T$. 

We now show that $Q=(\tilde{x}_1,x_2,\ldots,x_{n})$. Write $y\in Q$ as
\[y=\beta_{1,1}x_1+\cdots+\beta_{1,n}x_{n},\]
for some $\beta_{1,i}\in T$. Notice that $\tilde{x}_1\in Q$, since $x_1,y\in Q$; then,
\[\tilde{x}_1=x_1+\alpha_1 y=(1+\alpha_1\beta_{1,1})x_1+\alpha_1\beta_{1,2}x_2+\cdots+\alpha_1\beta_{1,n}x_{n}.\]
Rearranging, we have that
\[x_1=(1+\alpha_1\beta_{1,1})^{-1}(\tilde{x}_1-\alpha_1\beta_{1,2}x_2-\cdots-\alpha_1\beta_{1,n}x_{n})\in(\tilde{x}_1,x_2,\ldots,x_{n}),\]
where $(1+\alpha_1\beta_{1,1})$ is a unit because $\alpha_1\in M$. Thus, we are able to replace $x_1$ with $\tilde{x}_1$ in our generating set for $Q$. Notice that this argument works even if $n=1$.

To construct $R_3$, let $D_2=\bigcup_{P\in C}D_{(P)}$, where $D_{(P)}$ is a full set of coset representatives of the elements $t + P$ of $T/P$ that make $x_2+ty+P \in T/P$ algebraic over $R_2/(R_2\cap P)$. Note that $D_2$ is countable. Again using Lemma~\ref{lem:pippa_prime_avoidance}, there exists $\alpha_2\in M$ such that $x_2+\alpha_2y+P \in T/P$ is transcendental over $R_2/(R_2\cap P)$ for every $P\in C$. Let $\tilde{x}_2\coloneqq x_2+\alpha_2y$. Then, $R_3\coloneqq R_2[\tilde{x}_2]_{R_2[\tilde{x}_2]\cap M}$ is a $BR_0$-subring of $T$ by Lemma~\ref{lem:pb_subring_adjoin}. We have that $Q=(\tilde{x}_1,\tilde{x}_2,x_3,\ldots,x_{n})$ by a similar argument as above by writing $y=\beta_{2,1}\tilde{x}_1+\beta_{2,2}x_2+\cdots+\beta_{2,n}x_{n}$ to show that $x_2\in(\tilde{x}_1,\tilde{x}_2,x_3,\ldots,x_{n})$. 

Repeat the above process for each $i=4,\ldots,n+1$ to obtain a chain of $BR_0$-subrings of $T$, $R_1\subseteq\cdots\subseteq R_{n+1}$ and have $Q=(\tilde{x}_1,\tilde{x}_2,\ldots,\tilde{x}_{n})$. By construction, each $\tilde{x}_i\in R_{i+1}$, so $R_{n+1}$ contains a generating set for $Q$. Thus, $R'=R_{n+1}$ is our desired $BR_0$-subring of $T$.
\end{proof}

After adjoining generating sets for prime ideals via Lemma \ref{lem:adjoin_gen_set}, we obtain a $BR_0$-subring of $T$ whose completion is not necessarily $T$. If $R$ is a $BR_0$-subring of $T$, then our goal is to build a $BR_0$-subring $R'$ from $R$ such that $R\subseteq R'$, and $\widehat{R'}=T$. 

The following lemma, which is a modification of \cite[Lemma 2.6]{pippa}, is used to accomplish this goal. In particular, we use Lemma \ref{lem:pb_subring_fgideal_contain} to show in Lemma \ref{lem:construct_precompletion_from_pbsubring} that one can close up ideals of a $BR_0$-subring of $T$. We then use Corollary \ref{cor:basering_complete_macine} to show that the resulting ring is Noetherian and its completion is $T$. 

\begin{lemma}\label{lem:pb_subring_fgideal_contain}
Suppose $(T,M)$ is a complete local ring with $\dim T\geq 1$, $(R_0,R_0\cap M)$ is a countable local domain with $R_0\subseteq T$ and $\widehat{R}_0=T$, and $(R,R \cap M)$ is a $BR_0$-subring of $T$. Let $I$ be a finitely generated ideal of $R$ and let $c\in IT\cap R$. Then there exists a $BR_0$-subring $(R', R' \cap M)$ of $T$ such that $R\subseteq R'$ and $c\in IR'$.
\end{lemma}

\begin{proof}
We induct on the number of generators of $I$. For the base case, suppose $I=aR$. If $a=0$, then $c=0$, so $R$ is the desired $BR_0$-subring. Now consider the case where $a\neq 0$. Then, $c=au$ for some $u\in T$. We show that $R'\coloneqq R[u]_{(R[u]\cap M)}$ is the desired $BR_0$-subring. First, note that $R[u]_{(R[u]\cap M)}$ is countable, since $R$ is countable. Now suppose $f\in R[u]\cap P$ for $P\in\Ass(T)$. Then, $f=r_n u^n+\cdots+r_1u+r_0\in P$ for $r_i\in R$. Multiplying through by $a^n$, we get 
\[a^nf=r_n(au)^n+\cdots+r_1a^{n-1}(au)+r_0a^n=r_n c^n+\cdots+r_1a^{n-1}c+r_0a^n.\]
Notice that $r_i,c,a\in R$ for all $i$ so $a^nf$ is an element of $P\cap R=(0)$. Since $a\in R$ with $a\neq 0$, it cannot be that $a$ is a zero divisor in $T$. Thus, it must be that $f=0$, showing that $R[u]_{(R[u]\cap M)}$ is indeed a $BR_0$-subring of $T$.  Since $c = au$, we have that $c \in IR[u]_{(R[u]\cap M)} = IR'$.

Now suppose $I$ is generated by $m>1$ elements, and suppose that the lemma holds true for all ideals of $R$ generated by $m-1$ elements. Let $I=(y_1,\ldots,y_m)R$ where $y_i\neq 0$ for all $i=1,2,\ldots,m$. Then, $c=y_1t_1+\cdots+ y_mt_m$ for some $t_i\in T$. Note that
\[c=y_1t_1+(y_1y_2 t-y_1y_2 t)+y_2t_2+\cdots+y_mt_m=y_1(t_1+y_2t)+y_2(t_2-y_1t)+y_3t_3+\cdots+y_mt_m\]
for any $t\in T$. Let $x_1=t_1+y_2t$ and $x_2=t_2-y_1t$, where we will choose the element $t$ later. Now, let $P$ be a a maximal element of $\Ass(T)$. If $(t_1+y_2t)+P=(t_1+y_2t')+P$ for some $t,t'\in T$, then it must be the case that $y_2(t-t')\in P$. But $y_2\in R$ with $y_2\neq 0$ and $R\cap P=(0)$, so we must have that $t-t'\in P$. Thus, $t+P=t'+P$. The contrapositive of this result indicates that if $t+P\neq t'+P$, then $(t_1+y_2t)+P\neq (t_1+y_2t')+P$. Let $D_{(P)}$ be a full set of coset representatives of the cosets $t+P$ that make $x_1+P$ algebraic over $R/(R\cap P)$ in $T/P$. Since $R$ is countable, the algebraic closure of $R/(R\cap P)$ in $T/P$ is countable, so $D_{(P)}$ is countable. Let $C$ be the maximal elements of $\Ass(T)$ and let $D=\bigcup_{P\in C}D_{(P)}$. Note that $D$ is also countable, and, by Lemma~\ref{lem:t_mod_p_uncountable}, $T$ is uncountable.  Therefore, $|D| < |T|$. In addition, since $T$ is the completion of a local domain and $\dim T\geq 1$, we have by Theorem~\ref{thm:lech_char_domain} that $M\notin \Ass(T)$, and so $M\nsubseteq P$ for all $P\in C$. We now use Lemma~\ref{lem:pippa_prime_avoidance} with $C$ the maximal elements of $\Ass(T)$ and $I=M$ to find an element $t\in M\subseteq T$ such that $x_1+P \in T/P$ is transcendental over $R/(R\cap P)$ for every $P\in C$. By Lemma~\ref{lem:pb_subring_adjoin}, we have that $R''\coloneqq R[x_1]_{(R[x_1]\cap M)}$ is a $BR_0$-subring of $T$. Now let $J=(y_2,\ldots,y_m)R''$ and $c^*=c-y_1x_1$. Notice that $c,y_1x_1\in R''$, and 
\[c^*=(y_1(t_1+y_2t)+y_2(t_2-y_1t)+y_3t_3+\cdots+y_mt_m)-y_1x_1=y_2(t_2-y_1t)+y_3t_3+\cdots+y_mt_m,\]
because of how we have defined $x_1$. Thus, we have that $c^*\in JT\cap R''$.

We can now use our induction assumption to draw the conclusion that there exists a $BR_0$-subring $(R', R' \cap M)$ of $T$ such that $R''\subseteq R'\subseteq T$ and $c^*\in JR'$. Thus, $c^*=y_2r_2+\cdots+y_m r_m$ for some $r_i\in R'$. It follows that $c=y_1x_1+y_2r_2+\cdots+y_mr_m\in IR'$, and thus $R'$ is the desired $BR_0$-subring.
\end{proof}

The following lemma ensures that, given a $BR_0$-subring $R$ of $T$, one can construct a $BR_0$-subring from $R$ whose completion is $T$.

\begin{lemma}\label{lem:construct_precompletion_from_pbsubring}
Suppose $(T,M)$ is a complete local ring with $\dim T\geq 1$ and $(R_0,R_0\cap M)$ is a countable local domain with $R_0\subseteq T$ and $\widehat{R}_0=T$. Let $(R,R\cap M)$ be a $BR_0$-subring of $T$. Then there exists a $BR_0$-subring of $T$, $(R',R'\cap M)$, such that $R\subseteq R'\subseteq T$, and, if $I$ is a finitely generated ideal of $R'$, then $IT\cap R'= IR'$. Thus, $R'$ is Noetherian and $\widehat{R'}=T$.
\end{lemma}

\begin{proof}
Consider the set
\[\Omega=\{(I,c)\mid I\text{ is a finitely generated ideal of $R$, }c\in IT\cap R\}.\]
Since $R$ is countable, so is $\Omega$. Enumerate $\Omega$ and let $1$ denote its first element. Now, we recursively define an ascending chain of $BR_0$-subrings, beginning with $R_1=R$. Given the $k^{\text{th}}$ element $(I,c)$ with the $BR_0$-subring $R_k$ defined, we will construct $R_{k+1}$. Let $I=(a_1,\ldots,a_\ell)R$ for $a_i\in R$. Notice that $I'=(a_1,\ldots,a_\ell)R_k$ is a finitely generated ideal of $R_k$, and if $c\in IT\cap R$, then $c\in IT\cap R\subseteq I'T\cap R_k$, since $R=R_1\subseteq R_k$. Now let $R_{k+1}$ be the $BR_0$-subring obtained from Lemma~\ref{lem:pb_subring_fgideal_contain} so that $R_k\subseteq R_{k+1}\subseteq T$ and  $c\in I'R_{k+1}$. Then define $S_1=\bigcup_{i=1}^\infty R_i$. Since this is a countable union of $BR_0$-subrings, we have that $S_1$ is a $BR_0$-subring of $T$. By our construction, if $I$ is a finitely generated ideal of $R$ and $c\in IT\cap R$, then for some $i\in\Z^+$, we have that $c\in IR_i\subseteq IS_1$; thus, $IT\cap R\subseteq IS_1$.

We repeat this process using $S_1$ in place of $R$ to obtain a $BR_0$-subring $S_2$ such that $IT\cap S_1\subseteq IS_2$ for every finitely generated ideal $I$ of $S_1$. Continuing this process results in an ascending chain of $BR_0$-subrings $R\subseteq S_1\subseteq S_2\subseteq\cdots$ such that $IT\cap S_i\subseteq IS_{i+1}$ for every finitely generated ideal $I$ of $S_i$. 

Let $R'\coloneqq\bigcup_{i=1}^\infty S_i$. Note that $R'$ is a $BR_0$-subring of $T$. We show that, for every finitely generated ideal $I$ of $R'$, $IT\cap R'=I$. Consider some finitely generated ideal $I$ of $R'$. Clearly, $IR'\subseteq IT\cap R'$, so suppose that $c\in IT\cap R'$. Let $I=(a_1,\ldots,a_m)R'$. For some $k$, we have that $a_i\in S_k$ for all $i$ and $c\in S_k$. Then,
\[c\in (a_1,\ldots,a_m)T\cap S_k\subseteq (a_1,\ldots,a_m)S_{k+1}\subseteq IR'.\]
Thus we have $IT\cap R'=IR'$. 

Since $R_0\subseteq R\subseteq R'$ and for every finitely generated ideal $I$ of $R'$, we have that $IT\cap R'=IR'$, by Corollary~\ref{cor:basering_complete_macine}, we have that $R'$ is Noetherian and $\widehat{R'}=T$.
\end{proof}

We are now able to give sufficient conditions for a complete local ring containing the rationals and of dimension at least one to be the completion of a countable excellent local domain. The proof of Theorem \ref{thm:countable_excellent_domain_existence} is adapted from the proof of \cite[Theorem 3.5]{loepp2018uncountable}.

\begin{theorem}\label{thm:countable_excellent_domain_existence}
Let $(T,M)$ be a complete local ring containing the rationals with $\dim T\geq 1$.  Suppose that the following conditions are satisfied:
\begin{enumerate}
    \item $T$ is equidimensional,
    \item $T$ is reduced, and
    \item $T/M$ is countable.
\end{enumerate}
Then there exists a countable excellent local domain $(S,S\cap M)$ such that $S\subseteq T$ and $\widehat{S}=T$.
\end{theorem}

\begin{proof}
Since $T$ contains the rationals, all nonzero integers are units, and so no integer of $T$ is a zero divisor.
Because $T$ is reduced and $\dim T\geq 1$, we have that $M\notin\Ass(T)$; thus, $T$ is the completion of a countable local domain, $(R_0,R_0\cap M)$, by Theorem~\ref{prop:small19_countable_domain}.

Let $S_0=R_0$. Notice that $S_0$ is a local countable domain and has completion $T$. We define an ascending chain of rings recursively. For each $S_i$, we ensure that it contains $R_0$ and that it satisfies the criteria of Lemma~\ref{lem:Q_J_countable}, i.e., that $(S_i,S_i\cap M)$ is a countable local domain with $S_i\subseteq T$ and $\widehat{S_i}=T$. Suppose that $(S_i,S_i\cap M)$ is a countable local domain satisfying such properties. We construct $S_{i+1}$ to satisfy these as well. 

By Lemma~\ref{lem:Q_J_countable}, the set
\[\bigcup_{P\in\Spec(S_i)}\{Q\mid Q\in\min I\text{ for $I$ where }\Sing(T/PT)=V(I/PT)\}\]
is countable, and for any $Q$ in the set, we have that $Q\nsubseteq \frakp$ for all $\frakp\in\Ass(T)$.

Enumerate this set, $(Q_j)_{j\in\Z^+}$. Since $T$ is Noetherian, each $Q_j$ is finitely generated. We recursively define an ascending chain of $BR_0$-subrings of $T$, $S^*_0\subseteq S^*_1\subseteq \cdots$. Let $S^*_0=S_i$. For each successive $S^*_k$, we ensure that it contains a generating set of $Q_1,\ldots,Q_k$. 

First consider $S^*_0=S_i$ and $Q_1$. Since $S^*_0$ is a domain and $\widehat{S^*_0} = T$, we have that $S^*_0 \cap \frakp = (0)$ for all $\frakp \in \Ass (T)$.  It follows that $S^*_0$ is a $BR_0$-subring of $T$.  By Lemma~\ref{lem:adjoin_gen_set}, there exists a $BR_0$-subring $S_1^*$ of $T$ such that $S_0^*\subseteq S_1^*$ and $S_1^*$ contains a generating set for $Q_1$. Now suppose that $S^*_k$ has been defined, and it is a $BR_0$-subring of $T$ that contains generating sets for $Q_1,\ldots,Q_k$. Consider $Q_{k+1}$; we can again apply Lemma~\ref{lem:adjoin_gen_set} to find a $BR_0$-subring of $T$, $S^*_{k+1}$, such that $S^*_k\subseteq S^*_{k+1}$ and such that $S^*_{k+1}$ contains a generating set for $Q_{k+1}$.

Let $S'_i\coloneqq\bigcup_{k=0}^\infty S^*_k$. Notice that this is a countable union of ascending $BR_0$-subrings of $T$, so we have that $S'_i$ is a $BR_0$-subring of $T$.
Let $(S_{i+1}, S_{i+1}\cap M)$ be the $BR_0$-subring of $T$ obtained by applying Lemma~\ref{lem:construct_precompletion_from_pbsubring} to $S'_i$, so $(S_{i+1},S_{i+1}\cap M)$ is countable local domain such that $S'_i\subseteq S_{i+1}\subseteq T$ and $\widehat{S}_{i+1}=T$.
Thus, $S_{i+1}$ satisfies the conditions of Lemma~\ref{lem:Q_J_countable} as desired. This also implies by Proposition~\ref{prop:completion_machine_converse} that $IT\cap S_{i+1}=I$ for any finitely generated ideal $I$ of $S_{i+1}$. Note that, for any
\[Q\in\bigcup_{P\in\Spec(S_i)}\{Q\mid Q\in\min(I)\text{ for $I$ where }\Sing(T/PT)=V(I/PT)\},\]
we have that $S_{i+1}$ contains a generating set for $Q$.

Define $S\coloneqq\bigcup_{i=0}^\infty S_i$. We show that $(S,S\cap M)$ is a countable excellent local domain such that $R_0\subseteq S\subseteq T$ and $\widehat{S}=T$. First, since each $S_i$ a $BR_0$-subring of $T$, we have that $S$ is a $BR_0$-subring of $T$ as well so $(S,S\cap M)$ is a countable quasi-local domain with $S\subseteq T$. Note that, for every $i$, since $\widehat{S}_i=T$, if $I$ is a finitely generated ideal of $S_i$, then $IT\cap S_i=IS_i$. We show that this holds for $S$ as well.

Suppose that $I$ is a finitely generated ideal of $S$. We clearly have that $IS\subseteq IT\cap S$, so we will show that $IT\cap S\subseteq IS$. Since $I$ is finitely generated, $I=(a_1,\ldots,a_m)S$ for $a_i\in S$. Let $c\in IT\cap S$. Choose $\ell$ such that $a_i,c\in S_\ell$. Then,
\[c\in(a_1,\ldots,a_m)T\cap S_\ell=(a_1,\ldots,a_m)S_\ell\subseteq IS.\]
Thus, $IT\cap S\subseteq IS$. Since $R_0\subseteq S$, by Corollary~\ref{cor:basering_complete_macine}, $S$ is Noetherian and has completion $T$.

Finally, since $T$ is equidimensional, we can show that $S$ is excellent using Lemma~\ref{lem:excellent_sufficient_criteria}, i.e., by showing that, for $P\in\Spec(S)$ and $Q\in\Spec(T)$ such that $Q\cap S=P$, $(T/PT)_Q$ is a regular local ring. Let $P\in\Spec(S)$ and $Q\in\Spec(T)$ such that $Q\cap S=P$. Suppose for contradiction that $Q/PT\in\Sing(T/PT)=V(I/PT)$. Then, $Q\supseteq \frakp\supseteq I$ for some minimal prime ideal $\frakp\in\Spec(T)$ of $I$. Since $\frakp\supseteq I\supseteq PT$, $\frakp/PT\in\Sing(T/PT)$. Furthermore,
\[P=PT\cap S\subseteq \frakp\cap S\subseteq Q\cap S=P.\]
Thus, $\frakp \cap S=P$.

Note that $P$ is finitely generated, so let $P=(p_1,\ldots,p_m)S$. Choose $i$ so that $p_j\in S_i$ for all $j=1,\ldots,m$. Then define $P'\coloneqq P\cap S_i$, and note that $P'\in\Spec(S_i)$. We first show that $P'T=PT$. Observe that
\[P'=P\cap S_i=(PT\cap S)\cap S_i=PT\cap S_i=(p_1,\ldots,p_m)T\cap S_i=(p_1,\ldots,p_m)S_i,\]
so then
\[P'T=((p_1,\ldots,p_m)S_i)T=(p_1,\ldots,p_m)T=PT.\]
Thus, $T/PT=T/P'T$, so since $\frakp/PT\in\Sing(T/PT)$, we have $\frakp/P'T\in\Sing(T/P'T)$. Next, we show that $\frakp=PT$. Since $\frakp\cap S=P$, we have $\frakp\cap S_i=P\cap S_i=P'$. So then,
\[\frakp\in\bigcup_{P'\in\Spec(S_i)}\{\mathfrak{q}\mid \mathfrak{q}\in\min I\text{ for $I$ where }V(I/P'T)=\Sing(T/P'T)\}.\]
This means that, for some generating set for $\frakp$, $\{q_1,\ldots,q_\ell\}\subseteq T$, we have $\{q_1,\ldots,q_\ell\} \subseteq S_{i+1}$. Thus,
\[\frakp\cap S_{i+1}=(q_1,\ldots,q_\ell)T\cap S_{i+1}=(q_1,\ldots,q_\ell)S_{i+1}.\]
Since $S\supseteq S_{i+1}$, $\frakp\cap S\supseteq \frakp\cap S_{i+1}=(q_1,\ldots,q_\ell)S_{i+1}$. Recall that $\frakp\cap S=P$, so
\[PT=(\frakp\cap S)T\supseteq((q_1,\ldots,q_\ell)S_{i+1})T=(q_1,\ldots,q_\ell)T=\frakp.\]
We already know that $\frakp\supseteq PT$; thus $\frakp=PT$. Then, $(T/PT)_{\frakp}=(T/PT)_{PT}$ is a field, and therefore a regular local ring. However, $\frakp/PT\in\Sing (T/PT)$, so $(T/PT)_{\frakp}$ is not a regular local ring, which is a contradiction. It follows that $Q/PT\notin\Sing(T/PT)$, so $(T/PT)_Q$ is a regular local ring. Thus, by Lemma~\ref{lem:excellent_sufficient_criteria}, we have that $S$ is excellent.
\end{proof}

We are now ready to prove our main result.

\begin{theorem}\label{thm:countable_excellent_dimgeq1_char}
Let $(T,M)$ be a complete local ring with $\Q\subseteq T$. Then $T$ is the completion of a countable excellent local domain if and only if the following conditions hold:
\begin{enumerate}
    \item $T$ is equidimensional,
    \item $T$ is reduced, and
    \item $T/M$ is countable.
\end{enumerate}
\end{theorem}

\begin{proof}
First suppose $(T,M)$ is a complete local ring with $\dim T=0$ and $\Q\subseteq T$. Further suppose that $T$ is reduced, equidimensional, and $T/M$ is countable. Notice that $\bigcap_{P\in\Spec(T)} P=\sqrt{(0)}=(0)$, since the nilradical of a reduced ring is the $(0)$ ideal. But there is only one prime ideal in a local ring of dimension $0$, namely the maximal ideal. Thus, $M=(0)$ and $T$ is a field; then, $T/M=T/(0)=T$ is countable. In addition, $T$ itself is excellent, as it is a complete local ring. Thus, we have that $T$ is a countable excellent local domain, with completion itself.

Now suppose $(T,M)$ is a complete local ring with $\dim T\geq 1$ and $\Q\subseteq T$. In addition, suppose that $T$ is reduced, equidimensional, and $T/M$ is countable. By Theorem~\ref{thm:countable_excellent_domain_existence}, there exists a countable excellent local domain $S$ such that $\widehat{S}=T$. 

Suppose now that $T$ is the completion of a countable excellent local domain $A$ and $\dim T\geq 0$. By Theorem~\ref{prop:small19_countable_domain}, $T/M$ is countable. Notice that $A$ is universally catenary, since it is excellent. By \cite[Theorem 31.7]{matsumura}, this means that $A$ is formally catenary, i.e., $A/P$ is formally equidimensional for every $P\in\Spec(A)$. Since $A$ is a domain, we can consider $A/(0)\cong A$. Note that $A$ is formally equidimensional, so we have that $\widehat{A}=T$ is equidimensional. Finally, since $A$ is a domain, it is reduced. By \cite[Theorem 32.2]{matsumura}, its completion, $T$, is also reduced.
\end{proof}

As a corollary, we characterize complete local UFDs containing the rationals that are the completion of a countable excellent local UFD.

\begin{corollary}\label{cor:countable_excellent_UFD_char}
Let $(T,M)$ be a complete local UFD with $\Q\subseteq T$. Then $T$ is the completion of a countable excellent local UFD if and only if $T/M$ is countable.
\end{corollary}

\begin{proof}
Suppose $(T,M)$ is a complete local UFD with $\Q\subseteq T$ and suppose that $T/M$ is countable. Since $T$ is a domain, it is reduced and equidimensional. Then, by Theorem~\ref{thm:countable_excellent_dimgeq1_char}, $T$ is the completion of a countable excellent local domain, $(A,A\cap M)$. Since $T$ is a UFD, $A$ must be as well. Thus, $T$ is the completion of a countable excellent local UFD.

Now suppose $(T,M)$ is the completion of a countable excellent local UFD $(A,A\cap M)$. Then, $T$ is the completion of a countable local domain, so by Theorem~\ref{prop:small19_countable_domain}, $T/M$ is countable.
\end{proof}

\begin{example}
For some complete local rings, it is not difficult to find an example of a countable local excellent domain whose completion is the given complete local ring.  For example, the complete local ring $\Q[[x_1,\ldots,x_n]]$ is the completion of $\Q[x_1,\ldots,x_n]_{(x_1 \ldots ,x_n)}$. For other complete local rings, however, Theorem~\ref{thm:countable_excellent_dimgeq1_char} is more useful.  For example, perhaps surprisingly, we have by Theorem~\ref{thm:countable_excellent_dimgeq1_char} that the complete  local ring $\Q[[x,y,z]]/(xy)$ is the completion of a countable excellent local domain.
\end{example}

\section{Completions of Countable Noncatenary Local Domains and UFDs}\label{sec:noncat}

In this section, we give necessary and sufficient conditions for a complete local ring to be the completion of a countable noncatenary local domain, and neceesary and sufficient conditions for a complete local ring to be the completion of a countable noncatenary local UFD. To do this, we use some of the techniques developed in the previous section, especially those concerning $BR_0$-subrings.

We first focus on characterizing completions of countable noncatenary local domains. The following result from \cite{SMALL17} characterizes completions of noncatenary local domains.

\begin{theorem}[\cite{SMALL17}, Theorem 2.10]\label{thm:small17noncatdoman_char}
Let $(T,M)$ be a complete local ring. Then $T$ is the completion of a noncatenary local domain if and only if the following conditions hold:
\begin{enumerate}
    \item no integer of $T$ is a zero divisor,
    \item $M\notin\Ass(T)$, and
    \item there exists $P\in\Min(T)$ such that $1<\dim(T/P)<\dim T$.
\end{enumerate}
\end{theorem}

Using this result, along with Theorem~\ref{prop:small19_countable_domain}, we show that a complete local ring is the completion of a countable noncatenary local domain if and only if it is the completion of a noncatenary local domain and the completion of a countable local domain. In order to achieve this result, we first identify sufficient conditions needed by constructing a countable noncatenary local domain. We first provide some lemmas that will help us show that the rings we construct are noncatenary. 

\begin{lemma}[\cite{SMALL17}, Lemma 2.8]\label{lem:small17_saturated_chains}
Let $(T,M)$ be a local ring with $M\notin\Ass(T)$ and let $P\in\Min(T)$ with $\dim(T/P)=n$. Then there exists a saturated chain of prime ideals of $T$, $P\subsetneq Q_1\subsetneq \cdots\subsetneq Q_{n-1}\subsetneq M$, such that, for each $i=1,\ldots,n-1$, we have that $Q_i\notin\Ass(T)$ and $P$ is the only minimal prime ideal contained in $Q_i$.
\end{lemma}

We now use this result to prove the following lemma, which shows that under certain conditions, there exists a prime ideal with desirable properties.

\begin{lemma}\label{lem:prime_existence_codim1}
Let $(T,M)$ be a catenary local ring with $M\notin\Ass(T)$. If there exists $P\in\Min(T)$ such that $1<\dim(T/P)<\dim T$, then there exists $Q\in\Spec(T)$ such that $\dim(T/Q)=1$, $\height(Q)+\dim(T/Q)<\dim T$, and $Q\nsubseteq\frakp$ for all $\frakp\in\Ass(T)$.
\end{lemma}

\begin{proof}
Suppose $P\in\Min(T)$ and $1<\dim(T/P)=n<\dim T$. By Lemma~\ref{lem:small17_saturated_chains}, there exists a saturated chain of prime ideals of $T$, $P\subsetneq Q_1\subsetneq\cdots\subsetneq Q_{n-1}\subsetneq M$, such that, for each $i=1,\ldots,n-1$, we have that $Q_i\notin\Ass(T)$ and $P$ is the only minimal prime ideal of $T$ contained in $Q_i$. Since $T$ is catenary 
and this chain is saturated, we have that $\height(Q_i)+\dim(T/Q_i)=n<\dim(T)$ for all $i=1,\ldots,n-1$. In particular, if we consider $Q_{n-1}$, we have that $\dim(T/Q_{n-1})=1$ and $\height(Q_{n-1})+\dim(T/Q_{n-1})<\dim T$. In addition, we have that $M,Q_{n-1}\notin\Ass(T)$, so $Q_{n-1}\nsubseteq\frakp$ for all $\frakp\in\Ass(T)$ as well.
\end{proof}

We now identify sufficient conditions for a complete local ring to be the completion of a countable noncatenary local domain.

\begin{proposition}\label{prop:noncat_domain_suff}
Let $(T,M)$ be a complete local ring satisfying the following properties:
\begin{enumerate}
    \item no integer of $T$ is a zero divisor,
    \item $M\notin\Ass(T)$,
    \item there exists $P\in\Min(T)$ such that $1<\dim(T/P)<\dim T$, and
    \item $T/M$ is countable.
\end{enumerate}
Then $T$ is the completion of a countable noncatenary local domain.
\end{proposition}

\begin{proof}
Conditions (1), (2), and (4) imply that $T$ is the completion of a countable local domain, $(R_0,R_0\cap M),$ by Theorem~\ref{prop:small19_countable_domain}. Notice that $R_0$ is itself a $BR_0$-subring of $T$, and that $\dim T\geq 1$ by condition (3). In addition, $T$ is catenary since it is a complete local ring, and $T$ satisfies conditions (2) and (3); thus, there exists $Q\in\Spec(T)$ such that $\dim(T/Q)=1$, $\height(Q)+\dim(T/Q)<\dim T$, and $Q\nsubseteq\frakp$ for all $\frakp\in\Ass(T)$ by Lemma~\ref{lem:prime_existence_codim1}. Then, by Lemma~\ref{lem:adjoin_gen_set}, there exists a $BR_0$-subring of $T$, $(R,R \cap M)$, such that $R$ contains a generating set for $Q$. Since $\dim T\geq 1$, by Lemma~\ref{lem:construct_precompletion_from_pbsubring}, there exists a $BR_0$-subring of $T$, $(A,A\cap M)$, such that $R\subseteq A\subseteq T$, $A$ is Noetherian, and $\widehat{A}=T$. Since $A$ is a $BR_0$-subring of $T$, we have that $A$ is countable and a domain. It remains to be shown that $A$ is noncatenary.

Recall that $Q\in\Spec(T)$ such that $\dim(T/Q)=1$, $\height(Q)+\dim(T/Q)<\dim T$, and $A$ contains a generating set for $Q$. Since $A$ contains a generating set for $Q$, we have that $(Q\cap A)T=Q$. We show that $\dim(A/(Q\cap A))=1$ in order to show that $A$ is noncatenary. Suppose $P'$ is a prime ideal of $A$ such that $Q\cap A\subsetneq P'$. Then, $(Q\cap A)T=Q\subsetneq P'T$; otherwise, $Q=P'T$ would imply that $Q\cap A=P'T\cap A=P'$. Since $\dim(T/Q)=1$, we have that $\dim(T/P'T)=0$, which in turn implies that $\dim(A/P')=0$, since $\widehat{A/P'}=T/P'T$. Then, $P'$ is actually the maximal ideal of $A$, i.e., $P'=A\cap M$, and $\dim(A/(Q\cap A))=1$. A local ring and its completion satisfy the going-down theorem, so we have that $\height(Q\cap A)\leq\height(Q)$; then,
\[\height(Q\cap A)+1\leq\height(Q)+1<\dim T=\dim(A).\]
Thus, we have shown that $A$ is noncatenary.
\end{proof}

We now characterize completions of countable noncatenary local domains.

\begin{theorem}\label{thm:countable_noncat_domain_char}
Let $(T,M)$ be a complete local ring. Then $T$ is the completion of a countable noncatenary local domain if and only if the following conditions hold:
\begin{enumerate}
    \item no integer of $T$ is a zero divisor,
    \item $M\notin\Ass(T)$,
    \item there exists $P\in\Min(T)$ such that $1<\dim(T/P)<\dim T$, and
    \item $T/M$ is countable.
\end{enumerate}
\end{theorem}

\begin{proof}
If $T$ satisfies conditions (1), (2), (3), and (4), then by Proposition~\ref{prop:noncat_domain_suff}, $T$ is the completion of a countable noncatenary local domain. If $T$ is the completion of a countable noncatenary local domain, then by Theorem~\ref{thm:small17noncatdoman_char}, $T$ must satisfy conditions (1), (2), and (3). By Theorem~\ref{prop:small19_countable_domain}, $T$ must satisfy condition (4) as well.
\end{proof}

Next, we characterize completions of countable noncatenary local UFDs. The following is a result from \cite{SMALL17} that characterizes completions of noncatenary local UFDs.

\begin{theorem}[\cite{SMALL17}, Theorem 3.7]\label{thm:small17noncatufd_char}
Let $(T,M)$ be a complete local ring. Then $T$ is the completion of a noncatenary local UFD if and only if the following conditions hold:
\begin{enumerate}
    \item no integer of $T$ is a zero divisor,
    \item $\depth(T)>1$, and
    \item there exists $P\in\Min(T)$ such that $2<\dim(T/P)<\dim T$.
\end{enumerate}
\end{theorem}

We first identify sufficient conditions for a complete local ring to be the completion of a countable noncatenary local UFD. We do this by studying the construction of noncatenary local UFDs given in \cite{SMALL17}. The following result in \cite{SMALL17} identifies sufficient conditions on a complete local ring for it to be the completion of a noncatenary local UFD.

\begin{lemma}[\cite{SMALL17}, Lemma 3.6]\label{lem:small17_noncat_ufd_suff}
Let $(T,M)$ be a complete local ring such that no integer of $T$ is a zero divisor. Suppose $\depth T>1$ and there exists $P\in\Min(T)$ such that $2<\dim(T/P)<\dim T$. Then $T$ is the completion of a noncatenary local UFD.
\end{lemma}

In the proof of this characterization, the authors of \cite{SMALL17} use the proof of the following result of Heitmann's from \cite{heitmannUFD}.

\begin{theorem}[\cite{heitmannUFD}, Theorem 8]\label{thm:heitmann_UFD_suff}
Let $(T,M)$ be a complete local ring such that no integer is a zero divisor in $T$ and $\depth T\geq 2$. Then there exists a local UFD $A$ such that $\widehat{A}\cong T$ and $\abs{A}=\sup(\aleph_0,\abs{T/M})$. If $p\in M$ where $p$ is a nonzero prime integer, then $pA$ is a prime ideal.
\end{theorem}

In particular, the proof of Lemma~\ref{lem:small17_noncat_ufd_suff} is as follows. Given a complete local ring $(T,M)$ satisfying the conditions of the lemma's hypothesis, adjoin generators of a particular prime ideal of $T$ to the prime subring of $T$. Then use this ring as the base ring in the proof of Theorem~\ref{thm:heitmann_UFD_suff}, in which elements from a set of cardinality $\abs{T/M^2}$ are adjoined to the base ring, so that the resulting ring $A$ has cardinality $\sup(\aleph_0,\abs{T/M})$. The ring $A$ can then be shown to be a noncatenary local UFD with completion $T$.

Using this outline, we identify sufficient conditions on a complete local ring to be the completion of a countable noncatenary local UFD.

\begin{proposition}\label{prop:noncat_ufd_suff}
Let $(T,M)$ be a complete local ring such that the following conditions are satisfied:
\begin{enumerate}
    \item no integer of $T$ is a zero divisor,
    \item $\depth T>1$,
    \item there exists $P\in\Min(T)$ such that $2<\dim(T/P)<\dim T$, and
    \item $T/M$ is countable.
\end{enumerate}
Then $T$ is the completion of a countable noncatenary local UFD.
\end{proposition}

\begin{proof}
We show that given the added assumption that $T/M$ is countable, the ring constructed in the proof of Lemma~\ref{lem:small17_noncat_ufd_suff} is countable. The ring in the proof of this lemma is initialized by taking the prime subring of $T$ and adjoining generators of a prime ideal of $T$. The resulting ring, which we call $R$, is countable, since the prime subring of $T$ is countable, and adjoining finitely many elements to the prime subring yields a countable ring. Then, to prove Lemma \ref{lem:small17_noncat_ufd_suff}, the authors of \cite{SMALL17} use $R$ as the base ring in Heitmann's construction in the proof of Theorem~\ref{thm:heitmann_UFD_suff}. The resulting ring $A$ is countable, since $\abs{T/M^2}=\abs{T/M}$ by Proposition~\ref{prop:t/m^2_countable} so $A$ has cardinality $\sup(\aleph_0,\abs{T/M})=\aleph_0$. Finally, it is shown, in the proof of \cite[Lemma 3.6]{SMALL17}, that $A$ is a noncatenary local UFD with completion $T$.
\end{proof}

We are now able to characterize completions of countable noncatenary local UFDs.

\begin{theorem}\label{thm:noncat_ufd_char}
Let $(T,M)$ be a complete local ring. Then $T$ is the completion of a countable noncatenary local UFD if and only if the following conditions hold:
\begin{enumerate}
    \item no integer of $T$ is a zero divisor,
    \item $\depth T>1$,
    \item there exists $P\in\Min(T)$ such that $2<\dim(T/P)<\dim T$, and
    \item $T/M$ is countable.
\end{enumerate}
\end{theorem}

\begin{proof}
If $T$ satisfies conditions (1), (2), (3), and (4), then by Proposition~\ref{prop:noncat_ufd_suff}, $T$ is the completion of a countable noncatenary local UFD. If $T$ is the completion of a countable noncatenary local UFD, then by Theorem~\ref{thm:small17noncatufd_char}, $T$ must satisfy conditions (1), (2) and (3). By Theorem~\ref{prop:small19_countable_domain}, $T$ must satisfy condition (4) as well.
\end{proof}

\begin{example}
By Theorem~\ref{thm:countable_noncat_domain_char}, an example of a complete local ring that is the completion of a countable noncatenary local domain is $\Q[[x,y,z,w]]/(x)\cap(y,z)$. By Theorem~\ref{thm:noncat_ufd_char}, the complete local ring $\Q[x,y_1,y_2,z_1,z_2]]/(x)\cap(y_1,y_2)$ is the completion of a countable noncatenary local UFD.
\end{example}

\section*{Acknowledgments}
We thank the Clare Boothe Luce Scholarship Program for supporting the research of the second author.


\end{document}